\documentclass[runningheads]{casc}
\usepackage{amsmath}

\begin{document}

\setcounter{page}{363}

\newcommand{\C}{\textup{\bfseries C}}
\newcommand{\R}{\textup{\bfseries R}}
\newcommand{\machineeps}{\varepsilon_{\!\scriptscriptstyle M}}
\newcommand{\rem}{\textrm{rem}}
\newcommand{\res}{\textrm{res}}
\newcommand{\imi}{\textsl{i}}
\newcommand{\re}{\textrm{Re}}
\newcommand{\lc}{\textrm{\upshape lc}}
\newcommand{\prs}{\textrm{\upshape prs}}
\newcommand{\rprs}{\textrm{\upshape rprs}}
\newcommand{\syl}{\textrm{\upshape Syl}}
\newcommand{\subres}{\textrm{\upshape S}}
\newcommand{\recsubres}{\bar{\textrm{\upshape S}}} 

\title{Subresultants in Recursive Polynomial\\ Remainder Sequence%
  \thanks{This research was partially supported by Japanese Ministry
    of Education, Culture, Sports, Science and Technology under
    Grant-in-Aid for Young Scientists (B), 14780181, 2002.}}
\author{Akira Terui}
\institute{Institute of Mathematics\\
  University of Tsukuba\\
  Tsukuba, 305-8571 Japan\\
  \email{terui@math.tsukuba.ac.jp}}
\maketitle

\begin{abstract}
  We introduce concepts of ``recursive polynomial remainder sequence
  (PRS)'' and ``recursive subresultant,'' and investigate their
  properties.  In calculating PRS, if there exists the GCD (greatest
  common divisor) of initial polynomials, we calculate ``recursively''
  with new PRS for the GCD and its derivative, until a constant is
  derived.  We call such a PRS a \emph{recursive PRS}.  We define
  \emph{recursive subresultants} to be determinants representing the
  coefficients in recursive PRS by coefficients of initial
  polynomials.  Finally, we discuss usage of recursive subresultants
  in approximate algebraic computation, which motivates the present
  work.
\end{abstract}

\section{Introduction}
\label{sec:intro}

The polynomial remainder sequence (PRS) is one of fundamental tools in
computer algebra.  Although the Euclidean algorithm (see Knuth
(\cite{knu1998}) for example) for calculating PRS is simple,
coefficient growth in PRS makes the Euclidean algorithm often very
inefficient.  To overcome this problem, the mechanism of coefficient
growth has been extensively studied through the theory of
subresultants; see Collins (\cite{col1967}), Brown and Traub
(\cite{bro-tra71}), Loos (\cite{loos1983}), etc.  By the theory of
subresultant, we can remove extraneous factors of the elements of PRS
systematically.

In this paper, we consider a variation of the subresultant.  When we
calculate PRS for polynomials which have a nontrivial GCD, we usually
stop the calculation with the GCD.  However, it is sometimes useful to
continue the calculation by calculating the PRS for the GCD and its
derivative; this is necessary for calculating the number of real zeros
including their multiplicities.  We call such a PRS a ``recursive
PRS.''

Although the theory of subresultants has been developed widely, the
corresponding theory for recursive PRS is still unknown within the
author's knowledge: this is the main problem which we investigate in
this paper.  By ``recursive subresultants,'' we denote determinants
which represent elements of recursive PRS by the coefficients of
initial polynomials.

This paper is organized as follows.  In Sect.~\ref{sec:recprs}, we
introduce the concept of recursive PRS.  In
Sect.~\ref{sec:subresrecprs}, we define recursive subresultant and
show its relationship to recursive PRS.  In Sect.~\ref{sec:disc}, we
discuss briefly using recursive subresultants in approximate algebraic
computation.

\section{Recursive Polynomial Remainder Sequence (PRS)}
\label{sec:recprs}

First, we review the PRS, then define the recursive PRS.  At last, we
show recursive Sturm sequence as an example of recursive PRS.

\subsection{Definition of Recursive PRS}
\label{sec:defrecprs}

Let $R$ be an integral domain and polynomials $F$ and $G$ be in
$R[x]$.  We define a polynomial remainder sequence as follows.

\begin{definition}[Polynomial Remainder Sequence (PRS)]
  \label{def:prs}
  Let $F$ and $G$ be polynomials in $R[x]$ of degree $m$ and $n$
  ($m>n$), respectively.  A sequence
  \begin{equation}
    \label{eq:prs}
    (P_1,\ldots,P_l)
  \end{equation}
  of nonzero polynomials is called a \emph{polynomial remainder
    sequence (PRS)} for $F$ and $G$, abbreviated to $\prs(F,G)$, if it
  satisfies
  \begin{equation}
    \label{eq:prsdef}
    P_1=F,\quad P_2=G,\quad
    \alpha_i P_{i-2} = q_{i-1} P_{i-1} + \beta_i P_{i},
  \end{equation}
  for $i=3,\ldots,l$, where $\alpha_3,\ldots,\alpha_l,$
  $\beta_3,\ldots,\beta_l$ are elements of $R$ and
  $\deg(P_{i-1})>\deg(P_{i})$.  A sequence
  $((\alpha_3,\beta_3),\ldots,$ $(\alpha_l,\beta_l))$ is called a
  \emph{division rule} for $\prs(F,G)$ (see von zur Gathen
  and L\"ucking \textup{(\cite{vzg-luc2000})}).  If $P_l$ is a constant,
  then the PRS is called \emph{complete}.  \qed
\end{definition}

If $F$ and $G$ are coprime, the last element in the complete PRS for
$F$ and $G$ is a constant.  Otherwise, it equals the GCD of $F$ and
$G$ up to a constant: we have
$\prs(F,G)=(P_{1}=F,P_{2}=G,\ldots,P_{l}=\gamma\cdot \gcd(F,G))$ for
some $\gamma\in R$.  Then, we can calculate new PRS, $\prs(P_{l},
\frac{d}{dx}P_{l})$, and if this PRS ends with a non-constant
polynomial, then calculate another PRS for the last element, and so
on.  By repeating this calculation, we can calculate several PRSs
``recursively'' such that the last polynomial in the last sequence is
a constant.  Thus, we define ``recursive PRS'' as follows.
\begin{definition}[Recursive PRS]
  \label{def:recprs}
  Let $F$ and $G$ be the same as in Definition~\ref{def:prs}. Then, a
  sequence
  \begin{equation}
    \label{eq:recprs}
    (P_1^{(1)},\ldots,P_{l_1}^{(1)},
    P_1^{(2)},\ldots,P_{l_2}^{(2)},
    \ldots,
    P_1^{(t)},\ldots,P_{l_t}^{(t)})
  \end{equation}
  of nonzero polynomials is called a \emph{recursive polynomial
    remainder sequence} (recursive PRS) for $F$ and $G$, abbreviated
  to $\rprs(F,G)$, if it satisfies
  \begin{equation}
    \label{eq:recprsdef}
    \begin{split}
      & P_1^{(1)} = F,\quad P_2^{(1)}=G,\quad
      P_{l_1}^{(1)}=\gamma_1\cdot\gcd(P_1^{(1)},P_2^{(1)})\quad
      \mbox{with $\gamma_1\in R$}, \\
      & (P_1^{(1)},P_2^{(1)},\ldots,P_{l_1}^{(1)})=\prs(P_1^{(1)},P_2^{(1)}),\\
      & P_1^{(k)}=P_{l_{k-1}}^{(k-1)},\quad
      P_2^{(k)}=\frac{d}{dx}P_{l_{k-1}}^{(k-1)},\quad
      P_{l_k}^{(k)}=\gamma_k\cdot\gcd(P_1^{(k)},P_2^{(k)})\quad
      \mbox{with $\gamma_k\in R$}, \\
      & (P_1^{(k)},P_2^{(k)},\ldots,P_{l_k}^{(k)})=
      \prs(P_1^{(k)},P_2^{(k)}),
    \end{split}
  \end{equation}
  for $k=2,\ldots,t$.  If $\alpha_i^{(k)}$, $\beta_i^{(k)}\in R$
  satisfy
  \begin{equation}
    \alpha_i^{(k)} P_{i-2}^{(k)}
    =
    q_{i-1}^{(k)} P_{i-1}^{(k)}
    + 
    \beta_i^{(k)} P_i^{(k)}
  \end{equation}
  for $k=1,\ldots,t$ and $i=3,\ldots,l_k$, then a sequence
  $((\alpha_3^{(1)},\beta_3^{(1)}),\ldots,
  (\alpha_{l_t}^{(t)},\beta_{l_t}^{(t)}))$ is called a
  \emph{division rule} for $\rprs(F,G)$. 
  Furthermore, if $P_{l_t}^{(t)}$ is a constant, then the recursive
  PRS is called complete.
  \qed
\end{definition}

\subsection{Example of Recursive PRS: Recursive Sturm Sequence}
\label{sec:recsturmsec}

Sturm sequence is a variant of PRS, which is used in Sturm's method,
for calculating the number of real zeros of univariate polynomial (for
detail, see Cohen (\cite{coh93}) for example).  Note that Sturm's
theorem is valid for not only polynomials having simple zeros but also
those having multiple zeros (see Bochnak, Coste and Roy
(\cite{boc-cos-roy1998}) for example).  Here, we define ``recursive
Sturm sequence'' to calculate the number of real zeros including
multiplicities, as follows.

\begin{definition}[Recursive Sturm Sequence]
  Let $P(x)$ be a real polynomial of degree $m$.  Let a sequence of
  nonzero polynomials be defined by a recursive PRS in
  Definition~\ref{def:recprs}, calculated as
  \begin{equation}
    \label{eq:recsturmseq}
    \mathrm{(complete)}\; \rprs(P(x),\frac{d}{dx}P(x)),
  \end{equation}
  with division rule given by
  \begin{equation}
    (\alpha_i^{(k)},\beta_i^{(k)})=(1,-1),
  \end{equation}
  for $k=1,\ldots,t$ and $i=3,\ldots,l_k$.  Then, the sequence
  \textup{(\ref{eq:recsturmseq})} is called the \emph{recursive Sturm
  sequence} of $P(x)$. \qed
\end{definition}

\begin{example}[Recursive Sturm Sequence]
  \label{ex:recsturmseq}
  Let $P(x)=(x+2)^2\{(x-3)(x+1)\}^3$, and calculate the recursive
  Sturm sequence of $P(x)$.  The first sequence 
  $L_1=(P_1^{(1)},\ldots,P_4^{(1)})$ has the following elements:
  \begin{equation}
    \begin{split}
      P_1^{(1)} &= P(x)=(x+2)^2\{(x-3)(x+1)\}^3,\\
      P_2^{(1)} &=
      \frac{d}{dx}P(x)=8x^7-14x^6-102x^5+80x^4+460x^3+66x^2-558x-324,\\
      P_3^{(1)} &=
      \frac{75}{16}x^6-\frac{45}{16}x^5-60x^4-\frac{225}{8}x^3
      +\frac{3315}{16}x^2
      +\frac{4815}{16}x+\frac{945}{8},\\
      P_4^{(1)} &= \frac{128}{25}x^5-\frac{256}{25}x^4-\frac{256}{5}x^3
      +\frac{1024}{25}x^2+\frac{4224}{25}x+\frac{2304}{25}.
    \end{split}
  \end{equation}
  The second sequence $L_2=(P_1^{(2)},\ldots,P_4^{(2)})$ has the following
  elements: 
  \begin{equation}
    \begin{split}
      P_1^{(2)} &= P_4^{(1)} =
      \frac{128}{25}x^5-\frac{256}{25}x^4-\frac{256}{5}x^3
      +\frac{1024}{25}x^2
      +\frac{4224}{25}x+\frac{2304}{25},\\
      P_2^{(2)} &= \frac{d}{dx}P_4^{(1)} =
      \frac{128}{5}x^4-\frac{1024}{25}x^3-\frac{768}{5}x^2
      +\frac{2048}{25}x+\frac{4224}{25},\\
      P_3^{(2)} &= \frac{14848}{625}x^3-\frac{1536}{125}x^2
      -\frac{88576}{625}x-\frac{66048}{625},\\
      P_4^{(2)} &= \frac{12800}{841}x^2-\frac{25600}{841}x
      -\frac{38400}{841}.
    \end{split}
  \end{equation}
  The last sequence
  $L_3=(P_1^{(3)},\ldots,P_3^{(3)})$ has the following elements:
  \begin{equation}
    \begin{split}
      P_1^{(3)} &= P_4^{(2)} =
      \frac{12800}{841}x^2-\frac{25600}{841}x -\frac{38400}{841}, \\
      P_2^{(3)} &= \frac{d}{dx}P_4^{(2)} =
      \frac{25600}{841}x-\frac{25600}{841},\\
      P_3^{(3)} &= \frac{51200}{841}.
    \end{split}
  \end{equation}
  For PRS $L_k$, $k=1,2,3$, define sequences of nonzero real numbers
  $\lambda(L_k,-\infty)$ and $\lambda(L_k,+\infty)$ as
  \begin{equation}
    \begin{split}
      \lambda(L_k,-\infty) &= 
      ((-1)^{n_1^{(k)}}\lc(P_1^{(k)}),\ldots,
      (-1)^{n_{l_k}^{(k)}}\lc(P_{l_k}^{(k)})),\\ 
      \lambda(L_k,+\infty) &=
      (\lc(P_1^{(k)}),\ldots,
      \lc(P_{l_k}^{(k)})),
    \end{split}
  \end{equation}
  where $n_i^{(k)}=\deg(P_i^{(k)})$ denotes the degree of $P_i^{(k)}$
  and $\lc(P_i^{(k)})$ denotes the leading coefficients of
  $P_i^{(k)}$.  Then, $\lambda(L_k,-\infty)$ and
  $\lambda(L_k,+\infty)$ for $k=1,2,3$ are
  \begin{equation}
    \begin{split}
      \lambda(L_1,\pm\infty) &= (1,\pm8,\frac{75}{16},\pm\frac{128}{25}),\\
      \lambda(L_2,\pm\infty) &=
      (\pm\frac{128}{25},\frac{128}{5},\pm\frac{18848}{625},
      \frac{12800}{841}),\\  
      \lambda(L_3,\pm\infty) &=
      (\frac{12800}{841},\pm\frac{25600}{841},\frac{51200}{841}).
    \end{split}
  \end{equation}
  For a sequence of nonzero real numbers $L=(a_1,\ldots,a_m)$, let
  $V(L)$ be the number of sign variations of the elements of $L$.
  Then, we calculate the number of the real zeros of $P(x)$, including
  multiplicity, as
  \begin{equation}
    \sum_{k=1}^3\{V(\lambda(L_k,-\infty))-V(\lambda(L_k,+\infty))\}
    = 3+3+2=8.
  \end{equation}
  \qed
\end{example}

\section{Subresultants for Recursive PRS}
\label{sec:subresrecprs}

Let $F$ and $G$ be polynomials in $R[x]$ such that
\begin{equation}
  \label{eq:fg}
  F(x) = f_m x^m + \cdots + f_0 x^0,\quad
  G(x) = g_n x^n + \cdots + g_0 x^0,\quad
  \mbox{with $m\ge n>0$.}
\end{equation}
To make this paper self-contained and to use notations for our
definitions, we first review the fundamental theorem of subresultants,
then discuss subresultants for recursive PRS.

\subsection{Fundamental Theorem of Subresultants}

There exist several different definitions of subresultants.  Here, we
follow definitions and notations basically by von zur Gathen and
L\"ucking (\cite{vzg-luc2000}).

\begin{definition}[Sylvester Matrix]
  Let $F$ and $G$ be as in \textup{(\ref{eq:fg})}.
  The \emph{Sylvester matrix} of $F$ and $G$, denoted by $\syl(F,G)$,
  is an $(m+n)\times(m+n)$ matrix constructed from the coefficients of
  $F$ and $G$, such that
  \begin{equation}
    \label{eq:sylmat}
    \begin{split}
      \syl(F,G) &=
      \begin{pmatrix}
        f_m    &        &        & g_n    &        &  \\
        \vdots & \ddots &        & \vdots & \ddots &  \\
        f_0    &        & f_m    & g_0    &        & g_n \\
               & \ddots & \vdots &        & \ddots & \vdots \\
               &        & f_0    &        &        & g_0
      \end{pmatrix}.
      \\
      &\qquad \underbrace{\hspace{1.5cm}}_{n}
      \; \underbrace{\hspace{1.3cm}}_{m}
    \end{split}
  \end{equation}
  \qed
\end{definition}

Next, we define the ``subresultant matrix'' to derive polynomial
subresultants.

\begin{definition}[Subresultant Matrix]
  Let $F$ and $G$ be defined as in \textup{(\ref{eq:fg})}.  For $j<n$,
  the \emph{$j$-th subresultant matrix} of $F$ and $G$, denoted by
  $N^{(j)}(F,G)$, is an $(m+n-j)\times(m+n-2j)$ sub-matrix of
  $\syl(F,G)$ obtained by taking the left $n-j$ columns of
  coefficients of $F$ and the left $m-j$ columns of coefficients of
  $G$, such that
  \begin{equation}
    \label{eq:subresmat}
    \begin{split}
      N^{(j)}(F,G) &=
      \begin{pmatrix}
        f_m    &        &        & g_n    &        &  \\
        \vdots & \ddots &        & \vdots & \ddots &  \\
        f_0    &        & f_m    & g_0    &        & g_n \\
               & \ddots & \vdots &        & \ddots & \vdots \\
               &        & f_0    &        &        & g_0
       \end{pmatrix}.
      \\
      &\qquad \underbrace{\hspace{1.5cm}}_{n-j}
      \; \underbrace{\hspace{1.3cm}}_{m-j}
    \end{split}
  \end{equation}
  \qed
\end{definition}
\begin{definition}[Subresultant]
  Let $F$ and $G$ be defined as in \textup{(\ref{eq:fg})}.  For $j<n$
  and $k=0,\ldots,j$, let $N_k^{(j)}=N_k^{(j)}(F,G)$ be a sub-matrix
  of $N^{(j)}(F,G)$ obtained by taking the top $m+n-2j-1$ rows and the
  $(m+n-j-k)$-th row (note that $N^{(j)}_k(F,G)$ is a square matrix).
  Then, the polynomial
  \begin{equation}
    \label{eq:subres}
    \subres_j(F,G)
    =\det(N^{(j)}_j)x^j+\cdots+\det(N^{(j)}_0)x^0
  \end{equation}
  is called the \emph{$j$-th subresultant} of $F$ and $G$. \qed
\end{definition}
\begin{theorem}[Fundamental Theorem of Subresultants \cite{bro-tra71}]
  \label{th:fundsubres}
  Let $F$ and $G$ be defined as in \textup{(\ref{eq:fg})},
  $(P_1,\ldots,P_k)=\prs(F,G)$ be complete PRS, and
  $((\alpha_3,\beta_3),\ldots,$ $(\alpha_k,\beta_k))$ be its division
  rule.  Let $n_i=\deg(P_i)$ and $c_i=\lc(P_i)$ for $i=1,\ldots,k$,
  and $d_i=n_i-n_{i+1}$ for $i=1,\ldots,k-1$.  Then, we have
  \begin{align}
    \label{eq:subresthm1}
    \subres_j(F,G) &= 0 \quad\textrm{for $0\le j<n_k$},
    \\
    \subres_{n_i}(F,G) &= P_ic_i^{d_{i-1}-1}
    \prod_{l=3}^i
    \left\{
    \left(
      \frac{\beta_l}{\alpha_l}
    \right)^{n_{l-1}-n_i} 
    c_{l-1}^{d_{l-2}+d_{l-1}}
    (-1)^{(n_{l-2}-n_i)(n_{l-1}-n_i)}
    \right\}
    ,
    \\
    \subres_j(F,G) &= 0 \quad\textrm{for $n_i<j<n_{i-1}-1$},
    \\
    \label{eq:subresthm4}
    \subres_{n_{i-1}-1}(F,G) &= P_ic_{i-1}^{1-d_{i-1}}
    \prod_{l=3}^i
    \left\{\!
    \left(
      \frac{\beta_l}{\alpha_l}
    \right)^{n_{l-1}-n_{i-1}+1}\!\!\!
    c_{l-1}^{d_{l-2}+d_{l-1}}
    (-1)^{(n_{l-2}-n_{i-1}+1)(n_{l-1}-n_{i-1}+1)}\!
    \right\}
    ,
  \end{align}
  for $i=3,\ldots,k$. \qed 
\end{theorem}

By the above theorem, we can express coefficients of PRS by
determinants of matrices whose elements are the coefficients of
initial polynomials.

\subsection{Recursive Subresultants}

We construct ``recursive subresultant matrix'' whose determinants
represent elements of recursive PRS by the coefficients of initial
polynomials.

To make help for readers, we first show an example of recursive
subresultant matrix for recursive Sturm sequence in
Example~\ref{ex:recsturmseq}.
\begin{example}[Recursive Subresultant Matrix]
  \label{ex:recsubresmat}
  We express $P(x)$ and $\frac{d}{dx}P(x)$ in
  Example~\ref{ex:recsturmseq} by
  \begin{equation}
    P(x) = f_8 x^8 + \cdots + f_0 x^0,\quad
    \frac{d}{dx}F(x) = g_7 x^7 + \cdots + g_0 x^0.
  \end{equation}
  Let $M^{(1,5)}(F,G)=N^{(1,5)}(F,G)$, then
  the matrices $M_U^{(1,5)}(F,G)$, $M_L^{(1,5)}(F,G)$ and
  $M_L^{'(1,5)}(F,G)$ are given as
  \begin{equation}
    \label{eq:recsubresmatex}
      M^{(1,5)}(F,G)
      =
      \begin{pmatrix}
        M_U^{(1,5)} \\
        \hline
        M_L^{(1,5)}
      \end{pmatrix}
      =
      \begin{pmatrix}
        f_8  &      & g_7 &     &           \\
        f_7  & f_8  & g_6 & g_7 &           \\
        f_6  & f_7  & g_5 & g_6 & g_7 \\
        f_5  & f_6  & g_4 & g_5 & g_6 \\
        \hline
        f_4  & f_5  & g_3 & g_4 & g_5 \\
        f_3  & f_4  & g_2 & g_3 & g_4 \\
        f_2  & f_3  & g_1 & g_2 & g_3 \\
        f_1  & f_2  & g_0 & g_1 & g_2 \\
        f_0  & f_1  &     & g_0 & g_1 \\
        & f_0  &     &     & g_0 \\
      \end{pmatrix}
      ,
    \quad
      M_L^{'(1,5)}(F,G)
      =
      \begin{pmatrix}
        5f_4 & 5f_5 & 5g_3 & 5g_4 & 5g_5 \\
        4f_3 & 4f_4 & 4g_2 & 4g_3 & 4g_4 \\
        3f_2 & 3f_3 & 3g_1 & 3g_2 & 3g_3 \\
        2f_1 & 2f_2 & 2g_0 & 2g_1 & 2g_2 \\
        f_0  & f_1  &     & g_0 & g_1 \\
      \end{pmatrix}
      ,
  \end{equation}
  where horizontal lines in matrices divide them into the upper and
  the lower components.  Note that the matrix $M^{'(1,5)}(F,G)$ is
  derived from $M_L^{(1,5)}(F,G)$ by multiplying the $l$-th row by
  $6-l$ for $l=1,\ldots,5$ and deleting the lowest row.  Similarly,
  the $(2,3)$-th recursive subresultant matrix $M^{(2,3)}(F,G)$ is
  constructed as
  \begin{equation}
    \label{eq:recsubresmatex}
    M^{(2,3)}(F,G)=
    \left(
      \begin{array}{c|c|c}
        M_U^{(1,5)} &  &\\
        \cline{1-2}
                    & M_U^{(1,5)} & \\
                    \cline{2-3}
                    &             & M_U^{(1,5)}\\
                    \hline
                    &             &0 \cdots 0 \\
                    \cline{3-3}
        M_L^{(1,5)} & M_L^{'(1,5)} & \\
                    &             & M_L^{'(1,5)}\\
                    \cline{2-2}
                    & 0\cdots 0   &\\
      \end{array}
    \right).
  \end{equation}
  \qed
\end{example}

\begin{definition}[Recursive Subresultant Matrix]
  \label{def:recsubresmat}
  Let $F$ and $G$ be defined as in \textup{(\ref{eq:fg})}, and let
  $(P_1^{(1)},\ldots,P_{l_1}^{(1)},\ldots,P_1^{(t)},\ldots,P_{l_t}^{(t)})$
  be complete recursive PRS for $F$ and $G$ as in
  Definition~\ref{def:recprs}.  Let $j_0=m$ and $j_k=n_l^{(k)}$ for
  $k=1,\ldots,t$.  Then, for each tuple of numbers $(k,j)$ with
  $k=1,\ldots,t$ and $j=j_{k-1}-2,\ldots,0$, define matrix $M^{(k,j)}(F,G)$
  as follows.
  \begin{enumerate}
  \item For $k=1$, let $M^{(1,j)}(F,G)=N^{(j)}(F,G)$.
  \item For $k>1$, let $M^{(k,j)}(F,G)$ consist of the upper block and
    the lower block, defined as follows:
    \begin{enumerate}
    \item The upper block is partitioned into $(j_{k-1}-j_k-1)\times
      (j_{k-1}-j_k-1)$ blocks with diagonal blocks filled with
      $M_U^{(k-1,j_{k-1})}$, where $M_U^{(k-1,j_{k-1})}$ is a
      sub-matrix of $M^{(k-1,j_{k-1})}(F,G)$ obtained by deleting the
      bottom $j_{k-1}+1$ rows.
    \item Let $M_L^{(k-1,j_{k-1})}$ be a sub-matrix of
      $M^{(k-1,j_{k-1})}(F,G)$ obtained by taking the bottom
      $j_{k-1}+1$ rows, and let $M_L^{'(k-1,j_{k-1})}$ be a sub-matrix
      of $M_L^{(k-1,j_{k-1})}$ by multiplying the
      $(j_{k-1}+1-\tau)$-th rows by $\tau$ for
      $\tau=j_{k-1},\ldots,1$, then by deleting the bottom row.  Then,
      the lower block consists of $j_{k-1}-j-1$ blocks of
      $M_L^{(k-1,j_{k-1})}$ such that the leftmost block is placed at
      the top row of the container block and the right-side block is
      placed down by 1 row from the left-side block, then followed by
      $j_{k-1}-j$ blocks of $M_L^{'(k-1,j_{k-1})}$ placed by the same
      manner as $M_L^{(k-1,j_{k-1})}$.
    \end{enumerate}
  \end{enumerate}
\begin{figure}[t]
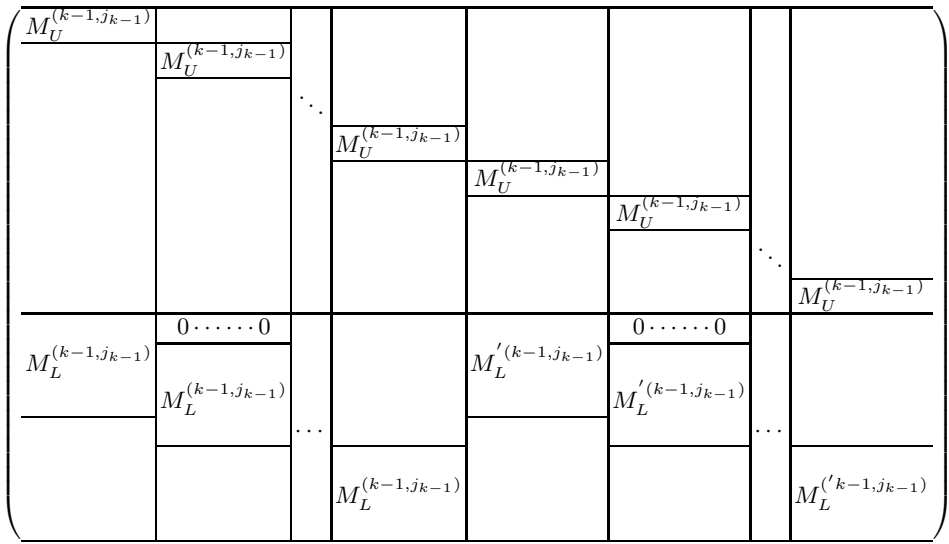

  \centering
  \begin{multline*}
    M^{(k,j)}(F,G)=
    \\
    \small
  \left(
  \begin{array}{c|c|c|c|c|c|c|c}
    \hline
    M_U^{(k-1,j_{k-1})} &            &        &          &
                  &            &        &         
     \\
     \cline{1-2}
             & M_U^{(k-1,j_{k-1})}   &        &          &
             &            &        &          
     \\
     \cline{2-2}
             &            & \ddots &          &
             &            &        &          
     \\
     \cline{4-4}
             &            &        & M_U^{(k-1,j_{k-1})} &
             &            &        &          
     \\
     \cline{4-5}
              &            &        &          &
     M_U^{(k-1,j_{k-1})} &            &        &          
     \\
     \cline{5-6}
              &            &        &          &
              & M_U^{(k-1,j_{k-1})}   &        &
     \\
     \cline{6-6}
              &            &        &          &
               &            & \ddots &         
     \\
     \cline{8-8}
               &            &        &         &
              &            &        & M_U^{(k-1,j_{k-1})}
     \\
     \hline
              & 0 \cdots\cdots 0 &        &          &
              & 0 \cdots\cdots 0 &        &          
     \\
     \cline{2-2}\cline{6-6}
     M_L^{(k-1,j_{k-1})} &            &        &          &
     M_L^{'(k-1,j_{k-1})} &            &        & 
     \\
              & M_L^{(k-1,j_{k-1})}   &        &     &
              & M_L^{'(k-1,j_{k-1})}   &        &          
    \\       
    \cline{1-1}\cline{5-5}
               &            & \cdots &          &
              &            & \cdots &          
    \\
    \cline{2-2}\cline{4-4}\cline{6-6}\cline{8-8}
               &            &        &          &
              &            &        &          
    \\
              &            &        & M_L^{(k-1,j_{k-1})} &
              &            &        & M_L^{('k-1,j_{k-1})}
    \\
              &            &        &          &
              &            &        &          
    \\
    \hline
  \end{array}
  \right).
  \end{multline*}
  \caption{Illustration of $M^{(k,j)}(F,G)$.  Note that
    the number of blocks $M_L^{(k-1,j_{k-1})}$ is $j_{k-1}-j-1$,
    whereas that of $M_L^{'(k-1,j_{k-1})}$ is $j_{k-1}-j$; see
    Definition~\ref{def:recsubresmat} for details.}
  \label{fig:recsubresmat}
\end{figure}

As a result, $M^{(k,j)}(F,G)$ becomes as shown in
Fig.~\ref{fig:recsubresmat}.  Then, $M^{(k,j)}(F,G)$ is called the
\emph{$(k,j)$-th recursive subresultant matrix} of $F$ and $G$. \qed
\end{definition}

\begin{proposition}
  \label{prop:recsubresmat}
  For $k=1,\ldots,t$ and $j<j_{k-1}-1$, the numbers of rows and
  columns of $M^{(k,j)}(F,G)$, the $(k,j)$-th recursive subresultant
  matrix of $F$ and $G$ are
  \begin{equation}
    \label{eq:recsubresmatrow}
    (m+n-2j_1)
    \left\{
      \prod_{l=2}^{k-1}(2j_{l-1}-2j_l-1)
    \right\}
    (2j_{k-1}-2j-1)
    +j
  \end{equation}
  and
  \begin{equation}
    \label{eq:recsubresmatcol}
    (m+n-2j_1)
    \left\{
      \prod_{l=2}^{k-1}(2j_{l-1}-2j_l-1)
    \right\}
    (2j_{k-1}-2j-1)
    ,
  \end{equation}
  respectively.
\end{proposition}
\begin{proof}
  By induction on $k$.  For $k=1$, by definition of the subresultant
  matrix, we see that $M^{(1,j)}(F,G)$ is a $(m+n-j)\times(m+n-2j)$
  matrix. Let us assume that equations
  \textup{(\ref{eq:recsubresmatrow})} and
  \textup{(\ref{eq:recsubresmatcol})} are valid for
  $1,\ldots,k-1$.  Then, we calculate the numbers of the rows
  and columns of $M^{(k,j)}(F,G)$ as follows.
\begin{enumerate}
\item The numbers of rows of $M_L^{(k-1,j_{k-1})}$ and
  $M_L^{'(k-1,j_{k-1})}$ are equal to $j_{k-1}+1$ and $j_{k-1}$,
  respectively, hence the number of rows a block which consists of
  $M_L^{(k-1,j_{k-1})}$ and $M_L^{'(k-1,j_{k-1})}$ in $M^{(k,j)}(F,G)$
  equals
  \begin{equation}
    \label{eq:recsubresmatrow-low}
    2j_{k-1}-j-1.
  \end{equation}
  On the other hand, the number of rows of $M_U^{(k-1,j_{k-1})}$ is
  equal to $(m+n-2j_1)\{\prod_{l=2}^{k-1}(2j_{l-1}-2j_l-1)\}-1$, hence
  the number of rows of diagonal blocks in $M^{(k,j_k)}(F,G)$ is equal
  to
  \begin{equation}
    \label{eq:recsubresmatrow-up}
    (m+n-2j_1)\left\{\prod_{l=2}^{k-1}(2j_{l-1}-2j_l-1)-1\right\}
    (2j_{k-1}-2j-1).
  \end{equation}
  By adding formulas (\ref{eq:recsubresmatrow-low}) and
  (\ref{eq:recsubresmatrow-up}), we obtain (\ref{eq:recsubresmatrow}).
\item The number of columns of $M^{(k-1,j_{k-1})}(F,G)$ is equal to
  $(m+n-2j_1)\{\prod_{l=2}^{k-1}(2j_{l-1}-2j_l-1)\}$, hence the number
  of columns of $M^{(k,j)}(F,G)$ is equal to
  (\ref{eq:recsubresmatcol}).
\end{enumerate}
This proves the proposition.
\qed
\end{proof}

Now, we define recursive subresultants by recursive subresultant
matrices.

\begin{definition}[Recursive Subresultant]
  \label{def:recsubres}
  Let $F$ and $G$ be defined as in \textup{(\ref{eq:fg})}, and let
  $(P_1^{(1)},\ldots,$
  $P_{l_1}^{(1)},\ldots,P_1^{(t)},\ldots,P_{l_t}^{(t)})$ be complete
  recursive PRS for $F$ and $G$ as in Definition~\ref{def:recprs}.
  Let $j_0=m$ and $j_k=n_l^{(k)}$ for $k=1,\ldots,t$.  For
  $j=j_{k-1}-2,\ldots,0$ and $\tau=j,\ldots,0$, let
  $M_\tau^{(k,j)}=M_\tau^{(k,j)}(F,G)$ be a sub-matrix of the
  $(k,j)$-th recursive subresultant matrix $M^{(k,j)}(F,G)$ obtained
  by taking the top
  $(m+n-2j_1)\{\prod_{l=2}^{k-1}(2j_{l-1}-2j_l-1)\}(2j_{k-1}-2j-1)-1$
  rows and the
  $((m+n-2j_1)\{\prod_{l=2}^{k-1}(2j_{l-1}-2j_l-1)\}(2j_{k-1}-2j-1)
  +j-\tau)$-th row (note that $M_\tau^{(k,j)}$ is a square matrix).
  Then, the polynomial
  \begin{equation}
    \recsubres_{k,j}(F,G)
    =\det(M_j^{(k,j)})x^j+\cdots+\det(M_0^{(k,j)})x^0
  \end{equation}
  is called the \emph{$(k,j)$-th recursive subresultant} of
  $F$ and $G$. \qed
\end{definition}

Finally, we derive the relation between recursive subresultants and
coefficients in recursive PRS.

\begin{lemma}
  \label{lem:recsubres}
  Let $F$ and $G$ be defined as in \textup{(\ref{eq:fg})}, and let
  $(P_1^{(1)},\ldots,P_{l_1}^{(1)},\ldots,P_1^{(t)},\ldots,P_{l_t}^{(t)})$
  be complete recursive PRS for $F$ and $G$ as in
  Definition~\ref{def:recprs}.  Let $c_i^{(k)}=\lc(P_i^{(k)})$,
  $n_i^{(k)}=\deg(P_i^{(k)})$, $j_0=m$ and $j_k=n_l^{(k)}$ for
  $k=1,\ldots,t$ and $i=1,\ldots,l_k$, and let
  $d_i^{(k)}=n_i^{(k)}-n_{i+1}^{(k)}$ for $k=1,\ldots,t$ and
  $i=1,\ldots,l_k-1$.  Furthermore, for $k=1,\ldots,t-1$ and
  $j=j_{k-1}-2,\ldots,0$, define 
  \begin{equation}
    \begin{split}
      u_{k,j} &= (m+n-2j_1)
      \left\{
        \prod_{l=2}^{k-1}(2j_{l-1}-2j_l-1)
      \right\}
      (2j_{k-1}-2j-1)
      ,
      \\
      B_k &= (c_{l_k}^{(k)})^{d_{l_k-1}^{(k)}-1}
      \prod_{l=3}^{l_k}\left\{
      \left(
        \frac{\beta_l^{(k)}}{\alpha_l^{(k)}}
      \right)^{n_{l-1}^{(k)}-n_{l_k}^{(k)}}
      (c_{l-1}^{(k)})^{(d_{l-2}^{(k)}+d_{l-1}^{(k)})}
      (-1)^{
        (n_{l-2}^{(k)}-n_{l_k}^{(k)})
        (n_{l-1}^{(k)}-n_{l_k}^{(k)})
      }
      \right\},
    \end{split}
  \end{equation}
  and let $u_k=u_{k,j_k}$.
  For $k=2,\ldots,t$ and $j=j_{k-1}-2,\ldots,0$, define
  \begin{equation}
    b_{k,j} = 2j_{k-1}-2j-1,\quad
    r_{k,j} = (-1)^{(u_{k-1}-1)(1+2+\cdots+(b_{k,j}-1))},
  \end{equation}
  and let $b_k=b_{k,j_k}$ and $r_k=r_{k,j_k}$.
  Then, for the $(k,j)$-th recursive subresultant of
  $F$ and $G$ with $k=1,\ldots,t$ and $j=j_{k-1}-2,\ldots,0$, we have
  \begin{equation}
    \label{eq:recsubreslem}
      \recsubres_{k,j}(F,G)
      =R_{k,j}\cdot\subres_j(P_1^{(k)},P_2^{(k)}),
  \end{equation}
  where $R_{1,j}=1$ and $R_{k,j}=((\cdots((B_1^{b_2}\cdot
  r_2B_2)^{b_3}\cdot r_3B_3)^{b_4}\cdots )^{b_{k-1}}\cdot
  r_{k-1}B_{k-1})^{b_{k,j}}\cdot r_{k,j}$ for $k>1$.
\end{lemma}
\begin{proof}
  For a univariate polynomial $P(x)=a_nx^n+\cdots+a_0x^0$ with $a_j\in
  R$ for $j=0,\ldots,n$, let us denote the coefficient vector for
  $P(x)$ by $\vec{p}={}^t(a_n,\ldots,a_0)$.
  
  We prove the lemma by induction on $k$.  For $k=1$, it follows
  immediately from the Fundamental Theorem of subresultants
  (Theorem~\ref{th:fundsubres}).  Let us assume that
  (\ref{eq:recsubreslem}) is valid for $1,\ldots,k-1$.  Then, we have
  \begin{equation}
    \recsubres_{k-1,j_{k-1}}(F,G)
      =R_{k-1,j_{k-1}}\cdot\subres_{j_{k-1}}(P_1^{(k-1)},P_2^{(k-1)}),
  \end{equation}
  hence we have
  \begin{equation}
    \det(M_\tau^{(k-1,j_{k-1})})=
    R_{k-1,j_{k-1}}\cdot\det(N_\tau^{(j_{k-1})}(P_1^{(k-1)},P_2^{(k-1)})),
  \end{equation}
  for $\tau=j_{k-1},\ldots,0$.  Therefore, there exists a matrix
  $W_{k-1}$ such that $\det(W_{k-1})=R_{k-1,j_{k-1}}$ and that we can
  transform $M^{(k-1,j_{k-1})}(F,G)$ to
  \begin{equation}
    \tilde{M}^{(k-1,j_{k-1})}(F,G)
    =
    \left(
      \begin{array}{c|c}
        W_{k-1} & \vec{O}\\
        \hline
        *       & N^{(j_{k-1})}(P_1^{(k-1)},P_2^{(k-1)})
      \end{array}
    \right),
  \end{equation}
  by eliminations on columns.  Furthermore, by eliminations and
  exchanges on columns\linebreak
  in the block
  $N^{(j_{k-1})}(P_1^{(k-1)},P_2^{(k-1)})$ as shown in Brown and Traub
  (\cite{bro-tra71}), we can transform\linebreak
  $\tilde{M}^{(k-1,j_{k-1})}(F,G)$ to
  \begin{equation}
    \bar{M}^{(k-1,j_{k-1})}(F,G)=
    \left(
      \begin{array}{c|c|c}
        W_{k-1} & \multicolumn{2}{c}{\vec{O}}\\
        \cline{1-2}
         & \bar{N}_U^{(j_{k-1})} &       \\
        \cline{2-3}
        \multicolumn{2}{l|}{\;\;*}
        & \vec{p}_1^{(k)}
      \end{array}
    \right),
  \end{equation}
  such that $\bar{N}_U^{(j_{k-1})}$ is a lower triangular matrix with
  all diagonal elements equal to 1\linebreak and that
  \mbox{$\det(\tilde{M}_\tau^{(k-1,j_{k-1})}(F,G))=
    B_{k-1}\cdot\det(\bar{M}_\tau^{(k-1,j_{k-1})}(F,G))$}, where
  $\tilde{M}_\tau^{(k-1,j_{k-1})}(F,G)$\linebreak and
  $\bar{M}_\tau^{(k-1,j_{k-1})}(F,G)$ are sub-matrices of
  $\tilde{M}^{(k-1,j_{k-1})}(F,G)$ and
  $\bar{M}^{(k-1,j_{k-1})}(F,G))$, respectively, obtained by the same
  manner as we have obtained $M_\tau^{(k-1,j_{k-1})}(F,G)$.
  Therefore, we have
  \begin{equation}
    \label{eq:mbardetk}
    \det(M_\tau^{(k-1,j_{k-1})}(F,G))=
    B_{k-1}\cdot\det(\bar{M}_\tau^{(k-1,j_{k-1})}(F,G)).
  \end{equation}
  Similarly, let $M^{'(k-1,j_{k-1})}(F,G)=
  \begin{pmatrix}
    M_U^{(k-1,j_{k-1})} \\
    \hline
    M_L^{'(k-1,j_{k-1})}
  \end{pmatrix}
  $.  Then, by the same transformations in the above, we can transform
  $M^{'(k-1,j_{k-1})}(F,G)$ to
  \begin{equation}
    \bar{M}^{'(k-1,j_{k-1})}(F,G)=
    \left(
      \begin{array}{c|c|c}
        W_{k-1} & \multicolumn{2}{c}{\vec{O}}\\
        \cline{1-2}
         & \bar{N}_U^{(j_{k-1})} &\\
        \cline{2-3}
        \multicolumn{2}{l|}{\;\;*}
        & \vec{p}_2^{(k)}
      \end{array}
    \right),
  \end{equation}
  satisfying
  \begin{equation}
    \label{eq:mbardetk-1}
    \det(M_\tau^{'(k-1,j_{k-1})}(F,G))=
    B_{k-1}\cdot\det(\bar{M}_\tau^{'(k-1,j_{k-1})}(F,G)),
  \end{equation}
  where $M_\tau^{'(k-1,j_{k-1})}(F,G)$ and
  $\bar{M}_\tau^{'(k-1,j_{k-1})}(F,G))$ are sub-matrices of
  $M^{'(k-1,j_{k-1})}(F,G)$ and\linebreak
  $\bar{M}^{'(k-1,j_{k-1})}(F,G))$, respectively, obtained by taking
  the top $u_{k-1}-1$ rows and the $(u_{k-1}+j_{k-1}-\tau)$-th row for
  $\tau=j_{k-1}\ldots,1$.  Therefore, for \mbox{$j<j_{k-1}-1$}, we can
  transform $M^{(k,j)}(F,G)$ to $\bar{M}^{(k,j)}(F,G)$ as shown in
  Fig.~\ref{fig:recsumbresmatbar} by eliminations and exchanges on
  columns in each column block, and let $\bar{M}_\tau^{(k,j)}(F,G)$ be
  sub-matrix of $\bar{M}^{(k,j)}(F,G)$ obtained by the same manner as
  we have obtained $M_\tau^{(k,j)}(F,G)$.
  \begin{figure}[tb]
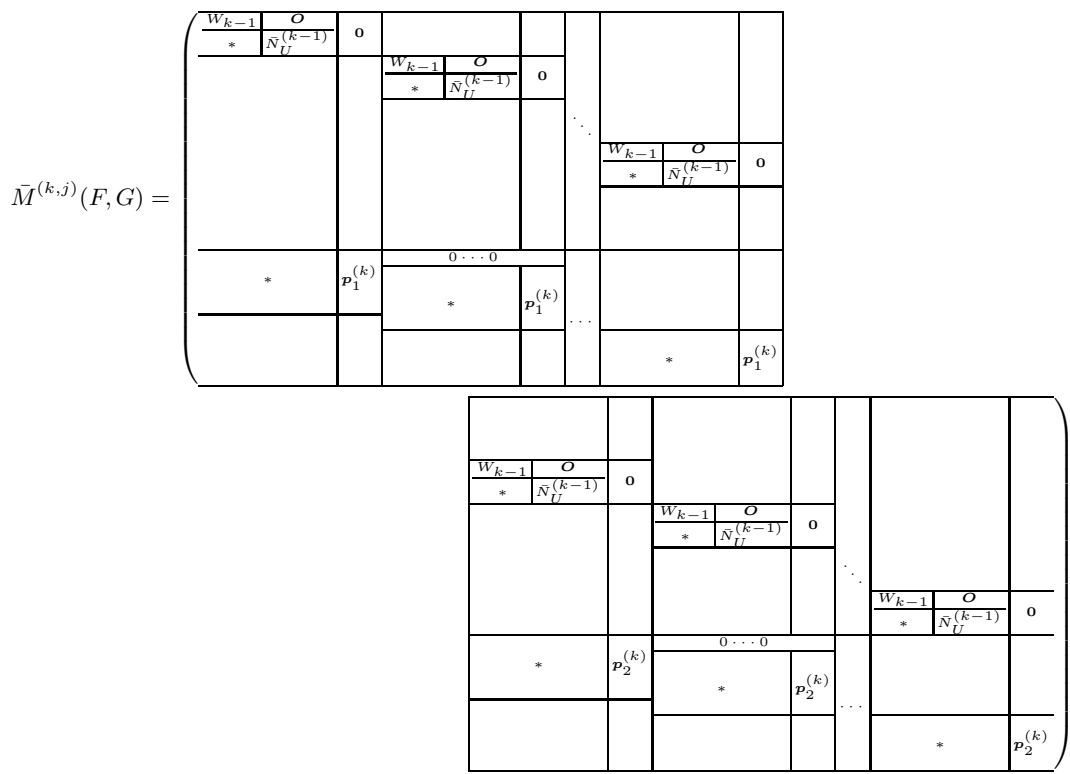

    \centering
  \begin{multline*}
    \bar{M}^{(k,j)}(F,G)=
    \left(
      \tiny
      \begin{array}{c|c|c|c|c|c|c|}
        \cline{1-7}
        \begin{array}{c|c}
          W_{k-1} & \vec{O}\\
          \hline
          * & \bar{N}_U^{(k-1)}
        \end{array}
        &
        \vec{0}
        & & & & &\\
        \cline{1-4}
        & &
        \begin{array}{c|c}
          W_{k-1} & \vec{O}\\
          \hline
          * & \bar{N}_U^{(k-1)}
        \end{array}
        &
        \vec{0}
        & & & \\
        \cline{3-4}
        & & & & \ddots & & \\
        \cline{6-7}
        & & & & &
        \begin{array}{c|c}
          W_{k-1} & \vec{O}\\
          \hline
          * & \bar{N}_U^{(k-1)}
        \end{array}
        &
        \vec{0}
        \\
        \cline{6-7}
        & & & & & &\\
        & & & & & &\\
        & & & & & &\\
        & & & & & &\\
        \cline{1-7}
        & & \multicolumn{2}{|c|}{0\cdots 0} & & &\\
        \cline{3-4}
        * &
        \vec{p}_1^{(k)} & & & & &\\
        & &
        * &
        \vec{p}_1^{(k)} & & &\\
        \cline{1-2}
        & & & & \cdots & &\\
        \cline{3-4}
        \cline{6-7}
        & & & & & &\\
        & & & & &
        * &
        \vec{p}_1^{(k)}\\
        & & & & & &\\
        \cline{1-7}
      \end{array}
    \right.
    \\
    \left.
      \tiny
      \begin{array}{|c|c|c|c|c|c|c}
        \cline{1-7}
        & & & & & & \\
        & & & & & & \\
        & & & & & & \\
        & & & & & & \\
        \cline{1-2}
        \begin{array}{c|c}
          W_{k-1} & \vec{O}\\
          \hline
          * & \bar{N}_U^{(k-1)}
        \end{array}
        &
        \vec{0}
        & & & & &\\
        \cline{1-4}
        & &
        \begin{array}{c|c}
          W_{k-1} & \vec{O}\\
          \hline
          * & \bar{N}_U^{(k-1)}
        \end{array}
        &
        \vec{0}
        & & & \\
        \cline{3-4}
        & & & & \ddots & & \\
        \cline{6-7}
        & & & & &
        \begin{array}{c|c}
          W_{k-1} & \vec{O}\\
          \hline
          * & \bar{N}_U^{(k-1)}
        \end{array}
        &
        \vec{0}
        \\
        \cline{1-7}
        & & \multicolumn{2}{|c|}{0\cdots 0} & & & \\
        \cline{3-4}
        * &
        \vec{p}_2^{(k)} & & &
        & & \\
        & &
        * &
        \vec{p}_2^{(k)} &
        & &  \\
        \cline{1-2}
        & & & & \cdots & & \\
        \cline{3-4}
        \cline{6-7}
        & & & & & & \\
        & & & & &
        * &
        \vec{p}_2^{(k)}\\
        & & & & & & \\
        \cline{1-7}
      \end{array}
    \right).
  \end{multline*}
    \caption{Illustration of
      $\bar{M}^{(k,j)}(F,G)$.  See Lemma~\ref{lem:recsubres} for
      details.}
    \label{fig:recsumbresmatbar}
  \end{figure}
  Then, we have
  \begin{equation}
    \label{eq:rel-m-mbar}
    \det(M_\tau^{(k,j)}(F,G))=
    (B_{k-1})^{b_{k,j}}\cdot\det(\bar{M}_\tau^{(k,j)}(F,G)),
  \end{equation}
  from (\ref{eq:mbardetk}) and (\ref{eq:mbardetk-1}) and since there
  exist $b_{k,j}$ blocks of $\bar{M}^{(k-1,j_{k-1})}(F,G)$ and
  $\bar{M}^{'(k-1,j_{k-1})}(F,G)$ in $\bar{M}^{(k,j)}(F,G)$ with each
  of which divided into the upper and the lower block.
  
  Furthermore, by exchanges on columns, we can transform
  $\bar{M}^{(k,j)}(F,G)$ to $\hat{M}^{(k,j)}(F,G)$ as shown in
  Fig.~\ref{fig:recsubresmathat}, and let $\hat{M}_\tau^{(k,j)}(F,G)$
  be sub-matrix of $\hat{M}^{(k,j)}(F,G)$ obtained by the same
  manner as we have obtained $M_\tau^{(k,j)}(F,G)$.
    \begin{figure}[h]
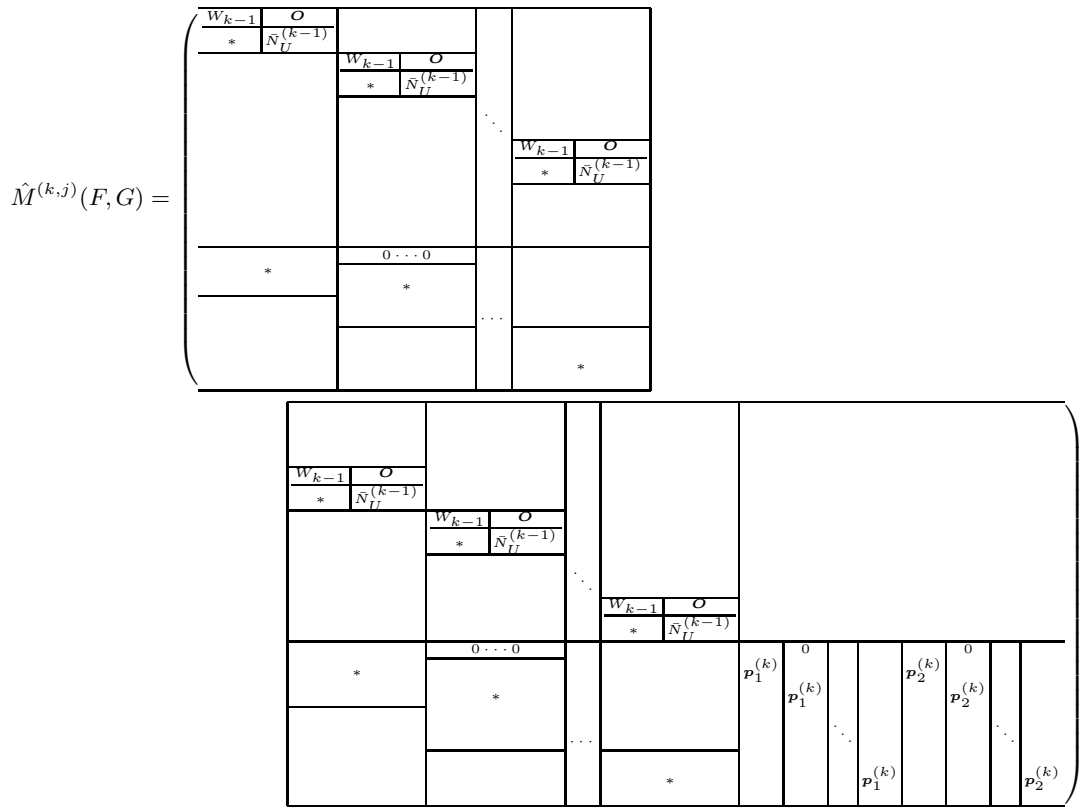

    \centering
    \begin{multline*}
      \hat{M}^{(k,j)}(F,G)=
      \tiny
      \left(
        \begin{array}{c|c|c|c|}
          \hline
          \begin{array}{c|c}
            W_{k-1} & \vec{O}\\
            \hline
            * & \bar{N}_U^{(k-1)}
          \end{array}
          & & &
          \\
          \cline{1-2}
          &
          \begin{array}{c|c}
            W_{k-1} & \vec{O}\\
            \hline
            * & \bar{N}_U^{(k-1)}
          \end{array}
          & &
          \\
          \cline{2-2}
          & & \ddots &
          \\
          \cline{4-4}
          & & &
          \begin{array}{c|c}
            W_{k-1} & \vec{O}\\
            \hline
            * & \bar{N}_U^{(k-1)}
          \end{array}
          \\
          \cline{4-4}
          & & &
          \\
          & & &
          \\
          & & &
          \\
          & & &
          \\
          \hline
          & 0\cdots 0 & &
          \\
          \cline{2-2}
          *
          & & &
          \\
          &
          *
          & &
          \\
          \cline{1-1}
          & & &
          \\
          & & \cdots &
          \\
          \cline{2-2}\cline{4-4}
          & & &
          \\
          & & &
          \\
          & & &
          *
          \\
          & & &
          \\
          \hline
        \end{array}
      \right.
      \\
      \tiny
      \left.
        \begin{array}{|c|c|c|c|c|c|c|c|c|c|c|c}
          \hline
          & & &\\
          & & &\\
          & & &\\
          & & & \\
          \cline{1-1}
          \begin{array}{c|c}
            W_{k-1} & \vec{O}\\
            \hline
            * & \bar{N}_U^{(k-1)}
          \end{array}
          & & &
          \\
          \cline{1-2}
          &
          \begin{array}{c|c}
            W_{k-1} & \vec{O}\\
            \hline
            * & \bar{N}_U^{(k-1)}
          \end{array}
          & &
          \\
          \cline{2-2}
          & & \ddots & 
          \\
          \cline{4-4}
          & & &
          \begin{array}{c|c}
            W_{k-1} & \vec{O}\\
            \hline
            * & \bar{N}_U^{(k-1)}
          \end{array}
          \\
          \hline
          & 0\cdots 0 & & & 
          & 0 & & &
          & 0 & &
          \\
          \cline{2-2}
          *
          & & & &  
          \vec{p}_1^{(k)} & & & &
          \vec{p}_2^{(k)} & & &
          \\
          & 
          *
          & & &
          & \vec{p}_1^{(k)} & & &
          & \vec{p}_2^{(k)} & &
          \\
          \cline{1-1}
          & & \cdots & & 
          & & \ddots & &
          & & \ddots &
          \\
          \cline{2-2}\cline{4-4}
          & & & &
          & & & &
          & & &
          \\
          & & &
          *
          & & & & \vec{p}_1^{(k)} &
          & & & \vec{p}_2^{(k)}
          \\
          & & & &
          & & & &
          & & &
          \\
          \hline
        \end{array}
      \right)
    \end{multline*}
    \caption{Illustration of $\hat{M}^{(k,j)}(F,G)$.  Note 
      that the lower-right block which consists of $\vec{p}_1^{(k)}$
      and $\vec{p}_2^{(k)}$ is equal to
      $N^{(j_k)}(P_1^{(k)},P_2^{(k)})$, and the number of blocks
      $W_{k-1}$ and $\bar{N}_{U}^{(k-1)}$ is $b_{k,j}=2j_{k-1}-2j-1$:
      see Lemma~\ref{lem:recsubres} for details.}
    \label{fig:recsubresmathat}
  \end{figure}
  Then, we have
  \begin{equation}
    \label{eq:rel-mbar-mhat}
    \det(\bar{M}_\tau^{(k,j)}(F,G))=
    r_{k,j}\cdot \det(\hat{M}_\tau^{(k,j)}(F,G)),
  \end{equation}
  because the $(u_{k,j}-(l-1)u_{k-1})$-th column in
  $\bar{M}^{(k,j)}(F,G)$ was moved to the $(u_{k,j}-(l-1))$-th column
  in $\hat{M}^{(k,j)}(F,G)$ for \mbox{$l=1,\ldots,b_{k,j}$}.
  Furthermore, we have
  \begin{equation}
    \label{eq:rel-mhat-subres}
    \det(\hat{M}_\tau^{(k,j)}(F,G))=
    (R_{k-1,j_{k-1}}B_{k-1})^{b_{k,j}}\cdot
    N_\tau^{(j)}(P_1^{(k)},P_2^{(k)}), 
  \end{equation}
  because the lower-right block of $\vec{p}_1^{(k)}$ and
  $\vec{p}_2^{(k)}$ in $\hat{M}^{(k,j)}(F,G)$ is equal to
  $N^{(j)}(P_1^{(k)},P_2^{(k)})$.

  Finally, from (\ref{eq:rel-m-mbar}), (\ref{eq:rel-mbar-mhat}) and
  (\ref{eq:rel-mhat-subres}), we have
  \begin{equation}
    \begin{split}
      \det(M_\tau^{(k,j)}(F,G)) &=
      r_{k,j}\cdot (R_{k-1,j_{k-1}}B_{k-1})^{b_{k,j}}\cdot
      \det(N_\tau^{(j)}(P_1^{(k)},P_2^{(k)}))
      \\
      &= R_{k,j}\cdot \det(N_\tau^{(j)}(P_1^{(k)},P_2^{(k)})).
    \end{split}
  \end{equation}
  Therefore, by the definitions of recursive subresultant, we obtain
  (\ref{eq:recsubreslem}).  This proves the lemma.  \qed
\end{proof}

\begin{theorem}
  With the same conditions as in  Lemma~\ref{lem:recsubres}, and for
  $k=1,\ldots,t$ and $i=3,4,\ldots,l_k$, we have
  \begin{align}
    \label{eq:recsubresthm1}
    &\recsubres_{k,j}(F,G) = 0 \quad\mbox{for $0\le j<n_{l_k}^{(k)}$},
    \\
    &\recsubres_{k,n_i^{(k)}}(F,G) = P_i^{(k)}
    (c_i^{(k)})^{d_{i-1}^{(k)}-1} R_{k,n_i^{(k)}}\nonumber
    \\
    & \qquad \times \prod_{l=3}^{i} \left\{ \left(
        \frac{\beta_l^{(k)}}{\alpha_l^{(k)}}
      \right)^{n_{l-1}^{(k)}-n_i^{(k)}}
      (c_{l-1}^{(k)})^{(d_{l-2}^{(k)}+d_{l-1}^{(k)})} (-1)^{
        (n_{l-2}^{(k)}-n_{i}^{(k)}) (n_{l-1}^{(k)}-n_{i}^{(k)}) }
    \right\},
    \\
    &\recsubres_{k,j}(F,G) = 0 \quad\mbox{for
      $n_i^{(k)}<j<n_{i-1}^{(k)}-1$},
    \\
    \label{eq:recsubresthm4}
    &\recsubres_{k,n_{i-1}^{(k)}-1}(F,G) = P_i^{(k)}
    (c_{i-1}^{(k)})^{1-d_{i-1}^{(k)}}R_{k,n_{i-1}^{(k)}-1} \nonumber\\
    & \qquad \times \prod_{l=3}^{i}
    \left\{
      \left(
        \frac{\beta_l^{(k)}}{\alpha_l^{(k)}}
      \right)^{n_{l-1}^{(k)}-n_{i-1}^{(k)}+1}
      (c_{l-1}^{(k)})^{(d_{l-2}^{(k)}+d_{l-1}^{(k)})} (-1)^{
        (n_{l-2}^{(k)}-n_{i-1}^{(k)}+1)
        (n_{l-1}^{(k)}-n_{i-1}^{(k)}+1) }
    \right\}.
  \end{align}
\end{theorem}
\begin{proof}
  By substituting $\subres_j(P_1^{(k)},P_2^{(k)})$ in
  (\ref{eq:recsubreslem}) by
  (\ref{eq:subresthm1})--(\ref{eq:subresthm4}), we obtain
  (\ref{eq:recsubresthm1})--(\ref{eq:recsubresthm4}), respectively.
  \qed
\end{proof}

We show an example of the proof of Lemma~\ref{lem:recsubres} for
recursive subresultant matrix in Example~\ref{ex:recsubresmat}.
\begin{example}
  Let us express $P_i^{(k)}$ in Example~\ref{ex:recsturmseq} by
  \begin{equation}
    P_i^{(k)}(x)=a_{i,n_i^{(k)}}^{(k)}x^{n_i^{(k)}}
    +\cdots+a_{i,0}^{(k)}x^0,
  \end{equation}
  with $n_i^{(k)}=\deg(P_i^{(k)})$.
  By eliminations and exchanges of columns as shown in Brown and
  Traub~(\cite{bro-tra71}), we can transform
  $M^{(1,5)}(F,G)=
  \begin{pmatrix}
    M_U^{(1,5)} \\
    \hline
    M_L^{(1,5)}
  \end{pmatrix}
  $ and
  $M^{'(1,5)}(F,G)=
  \begin{pmatrix}
    M_U^{(1,5)} \\
    \hline
    M_L^{'(1,5)}
  \end{pmatrix}
  $ in (\ref{eq:recsubresmatex}) to
  $\bar{M}^{(1,5)}(F,G)$ and $\bar{M}^{'(1,5)}(F,G)$, respectively, as
  \begin{equation}
    \begin{split}
      \bar{M}^{(1,5)}(F,G)
      &=
      \left(
        \begin{array}{c|c}
          \bar{N}_U^{(5)} & \vec{0} \\
          \hline
          *  & \vec{p}_1^{(2)}\\
        \end{array}
      \right)
      =
      \left(
      \begin{array}{cccc|c}
        1                   &                      &               &     &           \\
        \bar{a}_{2,6}^{(1)} & 1                    &               &     &           \\
        \bar{a}_{2,5}^{(1)} & \bar{a}_{2,6}^{(1)}  & 1             &     &    \\
        \bar{a}_{2,4}^{(1)} & \bar{a}_{2,5}^{(1)}  & \bar{a}_{3,5}^{(1)} & 1 & \\
        \hline
        \bar{a}_{2,3}^{(1)} & \bar{a}_{2,4}^{(1)}  & \bar{a}_{3,4}^{(1)} & \bar{a}_{3,5}^{(1)} &  a_{4,5}^{(1)} \\
        \bar{a}_{2,2}^{(1)} & \bar{a}_{2,3}^{(1)}  & \bar{a}_{3,3}^{(1)} & \bar{a}_{3,4}^{(1)} & a_{4,4}^{(1)} \\
        \bar{a}_{2,1}^{(1)} & \bar{a}_{2,2}^{(1)}  & \bar{a}_{3,2}^{(1)} & \bar{a}_{3,3}^{(1)} & a_{4,3}^{(1)} \\
        \bar{a}_{2,0}^{(1)} & \bar{a}_{2,1}^{(1)}  & \bar{a}_{3,1}^{(1)} & \bar{a}_{3,2}^{(1)} & a_{4,2}^{(1)} \\
                            & \bar{a}_{2,0}^{(1)}  & \bar{a}_{3,0}^{(1)} & \bar{a}_{3,1}^{(1)} & a_{4,1}^{(1)} \\
                            &                      &            & \bar{a}_{3,0}^{(1)} & a_{4,0}^{(1)} \\
      \end{array}
      \right)
      ,
      \\
      \bar{M}^{'(1,5)}(F,G) &=
      \left(
        \begin{array}{c|c}
          \bar{N}_U^{(5)} & \vec{0} \\
          \hline
          *  & \vec{p}_2^{(2)}\\
        \end{array}
      \right)
      =
      \left(
      \begin{array}{cccc|c}
        1                   &                &               &     &           \\
        \bar{a}_{2,6}^{(1)} & 1              &               &     &           \\
        \bar{a}_{2,5}^{(1)} & \bar{a}_{2,6}^{(1)}  & 1       &     &    \\
        \bar{a}_{2,4}^{(1)} & \bar{a}_{2,5}^{(1)}  & \bar{a}_{3,5}^{(1)} & 1 & \\
        \hline
        5\bar{a}_{2,3}^{(1)} & 5\bar{a}_{2,4}^{(1)} & 5\bar{a}_{3,4}^{(1)} & 5\bar{a}_{3,5}^{(1)} & 5a_{4,5}^{(1)} \\
        4\bar{a}_{2,2}^{(1)} & 4\bar{a}_{2,3}^{(1)} & 4\bar{a}_{3,3}^{(1)} & 4\bar{a}_{3,4}^{(1)} & 4a_{4,4}^{(1)} \\
        3\bar{a}_{2,1}^{(1)} & 3\bar{a}_{2,2}^{(1)} & 3\bar{a}_{3,2}^{(1)} & 3\bar{a}_{3,3}^{(1)} & 3a_{4,3}^{(1)} \\
        2\bar{a}_{2,0}^{(1)} & 2\bar{a}_{2,1}^{(1)} & 2\bar{a}_{3,1}^{(1)} & 2\bar{a}_{3,2}^{(1)} & 2a_{4,2}^{(1)} \\
                             & \bar{a}_{2,0}^{(1)}  & \bar{a}_{3,0}^{(1)}  & \bar{a}_{3,1}^{(1)} & a_{4,1}^{(1)} \\
      \end{array}
      \right),
    \end{split}
  \end{equation}
  where $\bar{a}_{i,j}^{(1)}=a_{i,j}^{(1)}/a_{2,7}^{(1)}$.
  Furthermore, we have 
  \begin{equation}
    \begin{split}
      \det(M_\tau^{(1,5)}(F,G)) &=
      B_1\cdot\det(\bar{M}_\tau^{(1,5)}(F,G))\quad
      \mbox{for $\tau=5,\ldots,0$},
      \\
      \det(M_\tau^{'(1,5)}(F,G)) &=
      B_1\cdot\det(\bar{M}_\tau^{'(1,5)}(F,G)) \quad
      \mbox{for $\tau=5,\ldots,1$},
    \end{split}
  \end{equation}
  with
  \begin{equation}
    B_1=-(a_{2,7}^{(1)})^2 (a_{3,6}^{(1)})^2,
  \end{equation}
  where $M_\tau^{(1,5)}(F,G)$ and $M_\tau^{'(1,5)}(F,G))$ are
  sub-matrices of $M^{(1,5)}(F,G)$ and $M^{'(1,5)}(F,G)$,
  respectively, obtained by taking the top 4 rows and the
  $(10-\tau)$-th row.  Therefore, by eliminations and exchanges on
  columns, we can transform $M^{(2,3)}(F,G)$ in
  (\ref{eq:recsubresmatex}) to $\bar{M}^{(2,3)}(F,G)$ as
  \begin{equation}
    \label{eq:recsubresmatbarex}
    \bar{M}^{(2,3)}(F,G)=
    \left(
      \begin{array}{c|c|c|c|c|c}
        \bar{N}_U^{(5)} & \vec{0} & & & &\\
        \cline{1-3}
                        &         &  \bar{N}_U^{(5)} & \vec{0} & &\\
        \cline{3-5}
                        &         &                  &         &
        \bar{N}_U^{(5)} & \vec{0} \\
        \hline
                        &         &                  &         & \multicolumn{2}{c}{0\cdots 0} \\
        \cline{5-6}
        *               & \vec{p}_1^{(2)} & * & \vec{p}_2^{(2)} &  &\\
                        &         &                  &          & * & \vec{p}_2^{(2)}\\
        \cline{3-4}
                        &         & \multicolumn{2}{c|}{0\cdots 0}
        & & \\
      \end{array}
    \right),
  \end{equation}
  satisfying
  $\det(M_\tau^{(2,3)}(F,G))=(B_1)^3\cdot\det(\bar{M}_\tau^{(2,3)}(F,G))$. 
  Furthermore, by exchanges on columns, we can transform
  $\bar{M}^{(2,3)}(F,G)$ to $\hat{M}^{(2,3)}(F,G)$ as
  \begin{equation}
    \label{eq:recsubresmathatex}
    \begin{split}
    \hat{M}^{(2,3)}(F,G) &=
    \left(
      \begin{array}{c|c|c|c|c|c}
        \bar{N}_U^{(5)}   \\
        \cline{1-2}
                        & \bar{N}_U^{(5)} \\
        \cline{2-3}
                        &      & \bar{N}_U^{(5)}    \\
        \hline
                        &         & {0\cdots 0}      &         & & 0  \\
        \cline{3-3}\cline{6-6}
        *               &  * & * & \vec{p}_1^{(2)} & \vec{p}_2^{(2)} &\\
                        &         &                  &          &  & \vec{p}_2^{(2)}\\
        \cline{2-2}\cline{5-5}
                        & {0\cdots 0}     &  &
        & 0 & \\
      \end{array}
    \right)
    \\
    &=
    \left(
      \begin{array}{c|c|c|c}
        \bar{N}_U^{(5)}   \\
        \cline{1-2}
                        & \bar{N}_U^{(5)} \\
        \cline{2-3}
        \multicolumn{2}{c|}{}  & \bar{N}_U^{(5)}    \\
        \hline
        \multicolumn{3}{c|}{*} & N^{(3)}(P_1^{(2)},P_2^{(2)})
      \end{array}
    \right),
    \end{split}
  \end{equation}
  satisfying
  $\det(\bar{M}_\tau^{(2,3)}(F,G))=r_{2,3}\cdot\det(\hat{M}_\tau^{(2,3)}(F,G))
  =r_{2,3}\cdot\det(N_\tau^{(3)}(P_1^{(2)},P_2^{(2)}))$.
  Therefore, we have
  \begin{equation}
    \det(M_\tau^{(2,3)}(F,G))=(B_1)^3r_{2,3}\cdot\det(N_\tau^{(3)}(P_1^{(2)},P_2^{(2)}))
    = R_{2,3}\cdot\det(N_\tau^{(3)}(P_1^{(2)},P_2^{(2)})),
  \end{equation}
  for $\tau=3,\ldots,0$, and we have
  \begin{equation}
    \recsubres_{2,3}(F,G)=R_{2,3}\cdot\subres_3(P_1^{(2)},P_2^{(2)})=
    \{(a_{2,7}^{(1)})^2(a_{3,6}^{(1)})^2\}^3(a_{2,4}^{(2)})^2\cdot P_3^{(2)}.
  \end{equation}
  \qed
\end{example}

\section{Conclusion and Motivation}
\label{sec:disc}

In this paper, we have defined recursive PRS as well as recursive
subresultants, and proved a similar theorem as the fundamental theorem
of subresultant.

The concept of recursive subresultant is inspired, in approximate
algebraic computation, by representing coefficients in recursive PRS
by the coefficients of initial polynomials.  For example, consider
calculating recursive Sturm sequence of a polynomial with
floating-point number coefficients by floating-point arithmetic.  In
the case the initial polynomial has multiple or close zeros, there may
exist a polynomial in the sequence such that it is difficult to decide
whether the polynomial becomes zero or not.  Also, zero recognition of
very small leading coefficient is another important problem because it
plays crucial role in calculating the number of real zeros.

For the problem of zero recognition of very small leading
coefficients, the present author and Sasaki (\cite{ter-sas2000}) have
proposed a criterion for calculating the number of real zeros
correctly by Sturm's method: if the Sturm sequence satisfy certain
condition on Sylvester matrix, then we can neglect the small leading
coefficient which makes computation of the Sturm sequence more stable.
We expect that the recursive subresultant (matrix) will be useful for
zero recognition of a polynomial in recursive Sturm sequence, by
representing its coefficients by the coefficients of initial
polynomials then analyzing it by numerical methods; this is the
problem on which we are working now.

\section*{Acknowledgments}

The author thank Prof.\ Tateaki Sasaki very much for revising the
original manuscript, and the referees for their helpful suggestions.


\end{document}